\documentclass[]{article}
\usepackage[left=2cm,right=2cm, top=2.4cm,bottom=2.4cm,bindingoffset=0cm]{geometry}
\usepackage{tikz, titlesec, tikz-cd}
\usetikzlibrary{calc,intersections,through,backgrounds}
\usepackage{amsmath, amssymb}
\usepackage{amsthm, mathtools}

\usepackage{array}
\newcolumntype{?}{!{\vrule width 1pt}}

\usepackage{ifthen}

\usepackage{faktor}

\mathtoolsset{showonlyrefs}

\usepackage{makecell}

\usepackage[numbers, sort&compress]{natbib}

\usepackage[bottom]{footmisc}

\usepackage{hyperref}

\definecolor{color1}{RGB}{127,201,127}
\definecolor{color2}{RGB}{190,174,212}
\definecolor{color3}{RGB}{253,192,134}
\definecolor{color4}{RGB}{255,255,153}

\usetikzlibrary{positioning}
\usetikzlibrary{decorations.text}
\usetikzlibrary{decorations.pathmorphing}
\usetikzlibrary{decorations.markings}

\usepackage{tkz-euclide}

\newtheorem{lemma}{Lemma}[section]

\newtheorem{theorem}[lemma]{Theorem}

\theoremstyle{definition}
\newtheorem{remark}[lemma]{Remark}

\newcommand{\R}{\mathbb{R}}
\renewcommand{\H}{\mathbb{H}}

\usepackage{bbm}

\usepackage{titlesec}

\titleformat*{\section}{\normalsize\bfseries}
\titleformat*{\subsection}{\Large\bfseries}
\titleformat*{\subsubsection}{\large\bfseries}
\titleformat*{\paragraph}{\large\bfseries}
\titleformat*{\subparagraph}{\large\bfseries}

\title{Change of polytope volumes under M\"{o}bius transformations and the circumcenter of mass}
\author{Anton Izosimov\thanks{
Department of Mathematics,
University of Arizona;
e-mail: {\tt izosimov@math.arizona.edu}
} }
\date{}
\begin{document}

\tikzset{->-/.style={decoration={
  markings,
  mark=at position .7 with {\arrow{>}}},postaction={decorate}}}
  
\usetikzlibrary{angles, quotes}

\maketitle

\abstract{
The circumcenter of mass of a simplicial polytope $P$ is defined as follows: triangulate $P$, assign to each simplex its circumcenter taken with
 weight equal to the volume of the simplex, and then find the center of mass
of the resulting system of point masses. The so obtained point is independent of the triangulation.

The aim of the present note is to give a definition of the circumcenter of mass that does not rely on a triangulation. To do so we investigate how volumes of polytopes change under M\"{o}bius transformations.



}

\section{Introduction}
Recall that the center of mass of a polyhedral solid $P$ can be found as follows: triangulate $P$, assign to each simplex its centroid taken with
 weight equal to the volume of the simplex, and then find the center of mass
of the resulting system of point masses. The so obtained point is independent of the triangulation.

Remarkably, when $P$ is simplicial (i.e. all its facets are simplices), one can replace centroids of simplices in this construction by their circumcenters. The resulting point still does not depend on the triangulation and is known as the \textit{circumcenter of mass} of $P$. The author of \cite{laisant1887theorie} attributes this construction to 19th century algebraic geometer G.\,Bellavitis. In modern literature, the circumcenter of mass is studied in \cite{TT, Akopyan, tabachnikov2015remarks, Adler, myakishev2006two}.

Since the circumcenter of mass does not depend on the triangulation, one should be able to define it without using one. The aim of the present note is to provide such a definition. Specifically, we show that the {circumcenter of mass} of $P$ is related to the rate of change of volume of $P$ under M\"{o}bius transformations. 

Recall that a M\"{o}bius transformation of $\R^n$ is an isometry of the hyperbolic upper half-space $\H^{n+1}$ restricted to the boundary. Every such transformation is a finite composition of inversions in spheres and reflections in hyperplanes. For $n \geq 3$ M\"{o}bius transformations are the same as conformal transformations, while for $n = 2$ M\"{o}bius transformations are just complex fractional linear transformations $z \mapsto (az + b)/(cz + d)$.

Note that general M\"{o}bius transformations do not preserve coplanarity and hence polyhedral shapes. However, the action of M\"{o}bius transformations on \textit{simplicial} polytopes is still well-defined. By definition, a M\"{o}bius transformation $\phi \colon \R^n \to \R^n$ takes a simplex $\Delta$ with vertices $v_0, \dots, v_n$ to a simplex $\phi(\Delta)$ with vertices $\phi(v_0), \dots, \phi(v_n)$. This extends to simplicial polytopes: the image under $\phi$ of a polytope with faces $\Delta_1, \dots, \Delta_k$ is the polytope with faces $\phi(\Delta_1), \dots, \phi(\Delta_k)$.

Consider a simplicial polytope $P$. How does its volume change under infinitesimal M\"{o}bius transformations? We show that this is determined by the location of one single point of $P$, that we call the \textit{M\"{o}bius center} of $P$ and denote as $\mathrm{m}(P)$  (see Theorem \ref{mainThm}, Part 1). Specifically, \textit{the relative rate of change of volume under an infinitesimal M\"{o}bius transformation $\xi$ is equal to the divergence of $\xi$ computed at the M\"{o}bius center}:
\begin{equation}\label{mainFormula}
\nabla_\xi \log \mathrm{vol}(P) = \mathrm{div}\, \xi(\mathrm{m}(P)).
\end{equation}
Here $\nabla_\xi $ stands for the derivative in the direction $\xi$:
$
\nabla_\xi \log \mathrm{vol}(P) =\left. \frac{d}{dt} \right\vert_{t=0} \log \mathrm{vol}(\phi_t(P)),
$
where $\phi_t$ is a family of M\"{o}bius transformations integrating $\xi$.

Note that M\"{o}bius vector field are quadratic and hence have linear divergence.  So, formula \eqref{mainFormula} implies that, just like the center and circumcenter of mass, the M\"{o}bius center $\mathrm{m}(P)$ can be found by subdivision into simplices   (see Theorem \ref{mainThm}, Part 3). Furthermore, we show that for a simplex $\Delta$ its M\"{o}bius center $\mathrm{m}(\Delta)$ coincides with the circumcenter of its medial simplex $\Delta'$, i.e. the simplex whose vertices are centroids of faces of $\Delta$  (see Theorem \ref{mainThm}, Part 4). So, the M\"{o}bius center can be defined by the same construction as the circumcenter of mass, with circumcenters of simplices of the triangulation replaced by circumcenters of their medial simplices. Furthermore, there is a simple relation between the two circumcenters: for a $n$-dimensional simplex $\Delta$ one has
\begin{equation}\label{proportion}
\mathrm{cc}(\Delta') = \frac{n+1}{n}\,\mathrm{cm}(\Delta) -  \frac{1}{n}\,\mathrm{cc}(\Delta), 
\end{equation}
where $\Delta'$ is the medial simplex, $\mathrm{cm}$ stands for the center of mass, and $\mathrm{cc}$ for the circumcenter. As a result, for a general simplicial $n$-dimensional polytope $P$ one has
$$
\mathrm{m}(P) = \frac{n+1}{n}\,\mathrm{cm}(P) -  \frac{1}{n}\,\mathrm{ccm}(P), 
$$
where $\mathrm{ccm}(P)$ is the circumcenter of mass of $P$. So, since both the center of mass and the M\"{o}bius center can be defined without a triangulation, it follows that the circumcenter of mass is well-defined as well. Explicitly, one has
\begin{equation}\label{finalFla}
\mathrm{ccm}(P) = (n+1)\,\mathrm{cm}(P) -  n \cdot \mathrm{m}(P). 
\end{equation}

\begin{remark}
One may similarly ask how the volume of a polytope changes under infinitesimal \textit{projective} transformations. In that case, one has the following version of formula \eqref{mainFormula}:
$
\nabla_\xi \log \mathrm{vol}(P) = \mathrm{div}\, \xi (\mathrm{cm}(P)).
$
Indeed, since projective transformations take faces to faces, the change of volume of $P$ under an infinitesimal {projective} transformation $\xi$ can be computed as the flux of $\xi$ through  the boundary of $P$ or, equivalently, as the integral of the divergence of $\xi$ over the interior of $P$ (for simplicity assume that $P$ is convex): 

 $$
\nabla_\xi \log \mathrm{vol}(P)  = \frac{1}{ \mathrm{vol}(P) } \int_{\mbox{interior of } P} \mathrm{div}\, \xi \, dx,
$$
where $dx$ is the Euclidean volume element. 
But since the divergence of a projective vector field is a (inhomogeneous) linear function, the latter expression is precisely the value of the divergence at the center of mass, as needed.
\end{remark}
\begin{remark} Formula \eqref{finalFla} may seem unsettling as the coefficients look pretty random. There is a way to fix this by replacing M\"{o}bius  vector fields
by a different class of quadratic fields $\xi$ which in some sense interpolate between M\"{o}bius and projective fields (see Remark~\ref{rm:vfccm} below). Those vector fields have the property
$
\nabla_\xi \log \mathrm{vol}(P) = \mathrm{div}\, \xi (\mathrm{ccm}(P))
$
and thus provide a direct definition of the circumcenter of mass circumventing the notion of the M\"{o}bius center. However, the geometric meaning of such fields $\xi$ is somewhat unclear. 


\end{remark}
\begin{remark}
For a triangle, the circumcenter of the medial triangle (i.e. the M\"{o}bius center) is also known as the \textit{nine-point center}. So, analogously to the definition of the circumcenter of mass, the  M\"{o}bius center of a polygon can be thought of as the ``nine-point center of mass''.
An analogous point of the tetrahedron is the center of the so-called \textit{twelve-point sphere}, however it does not seem to have any name. Relation \eqref{proportion} is well known in those cases. In particular, for a triangle it says that the centroid lies on the line joining the circumcenter and the nine-point center, $2/3$ of the way towards the latter (the line containing all the three points is known as the \textit{Euler line}; it also contains the orthocenter). 
\end{remark}
\smallskip

{\bf Acknowledgments.} The author is grateful to Boris Khesin, Leonid Monin, Richard Schwartz, and Sergei Tabachnikov for fruitful discussions and useful remarks. This work was supported by NSF grant DMS-2008021.

\section{Precise definitions and the main result}\label{sec2}

There are many ways to formalize the notion of a (not necessarily convex) polytope. For the purposes of the present paper, a \textit{simplicial polytope} in $\R^n$ is a piecewise linear simplicial cycle of dimension $n-1$ (a particular case of this general definition is the boundary of a convex polytope all of whose facets are simplices). In particular, polytopes in $\R^n$ form an Abelian group $\mathcal P(\R^n)$ under addition. This group is generated by boundaries of oriented $n$-dimensional simplices (in what follows, we refer to those generators as just simplices). A representation of a polytope $P$ as a sum of simplices is called a \textit{triangulation}. 

The group of polytopes comes  equipped with the (algebraic) volume homomorphism $\mathrm{vol} \colon \mathcal P(\R^n) \to \R$. It is given on simplices $\Delta = (v_0, \dots, v_n)$ by
$
\mathrm{vol}(\Delta) = \frac{1}{n!} |v_1 - v_0, \dots, v_n - v_0|,
$
where $|w_1, \dots, w_n|$ stands for the determinant of the matrix $(w_1, \dots, w_n)$.

We define \textit{M\"{o}bius transformations} of $\R^n$ as isometries of the hyperbolic upper half-space $\H^{n+1}$ restricted to the boundary. We refer the reader to \cite{ahlfors1981mobius} for a detailed account of such transformations. Here we only need the corresponding Lie algebra $\mathfrak{m\ddot ob}_n$ of M\"{o}bius vector fields. In dimensions $n \geq 3$, M\"{o}bius vector fields are the same as conformal Killing vector fields. In dimension $n = 2$, they are the same as holomorphic quadratic vector fields. In any dimension, the general form of a M\"{o}bius vector field $\xi$ is
\begin{equation}\label{eqn:mob}
\dot x = Ax  + |x|^2b - 2\langle b, x \rangle x + c,
\end{equation}
where $A $ is a matrix such that $A - \lambda \mathrm{Id}$ is skew-symmetric for some $\lambda \in \R$, and $b,c \in \R^n$ are vectors. We note that the divergence of such a vector field is
\begin{equation}\label{eqn:mobdiv}
\mathrm{div}\, \xi = \mathrm{tr}\, A - 2n \langle b, x \rangle.
\end{equation}
\begin{theorem}\label{mainThm}

\begin{enumerate} \item
Let $P \in \mathcal P(\R^n)$ be a polytope in $\R^n$ with $\mathrm{vol}(P) \neq 0$. Then there exists a unique point $\mathrm{m}(P)\in\R^n$ (the \textit{M\"{o}bius center of P}) such that
\begin{equation}\label{mainFormula2}
\nabla_\xi \log \mathrm{vol}(P) = \mathrm{div}\, \xi (\mathrm{m}(P)).
\end{equation}
for any M\"{o}bius vector field $\xi \in \mathfrak{m\ddot ob}_n$.
\item For $P$ as above and any similarity transformation $\phi \colon \R^n \to \R^n$ (i.e. a composition of a homothety and isometry) one has $\mathrm{m}(\phi(P))) = \phi(\mathrm{m}(P))$.
\item For $P$ as above and any triangulation $P = \sum \Delta_i$ with $\mathrm{vol}(\Delta_i)  \neq 0$ for all $i$ one has
$$
\mathrm{m}(P) = \frac{1}{\mathrm{vol}(P)} \sum \mathrm{vol}(\Delta_i) \, \mathrm{m}(\Delta_i). 
$$
\item For a simplex $\Delta \in \mathcal P(\R^n)$ such that $ \mathrm{vol}(\Delta) \neq 0$, the M\"{o}bius center $\mathrm{m}(\Delta)$ coincides with the circumcenter of the medial simplex $\Delta'$.
 It is related to the centroid and the circumcenter of $\Delta$ by the formula
 \begin{equation}\label{mcccm}
\mathrm{cc}(\Delta') = \frac{n+1}{n}\,\mathrm{cm}(\Delta) -  \frac{1}{n}\,\mathrm{cc}(\Delta).
\end{equation}
\item For any $P \in \mathcal P(\R^n)$ with $\mathrm{vol}(P) \neq 0$ its M\"{o}bius center is related to the center of mass and the circumcenter of mass by the formula
$$
\mathrm{m}(P) = \frac{n+1}{n}\,\mathrm{cm}(P) -  \frac{1}{n}\,\mathrm{ccm}(P), 
$$
\end{enumerate}
\end{theorem}
\begin{proof}
1. Consider the subalgebra $\mathfrak{iso}_n \subset \mathfrak{m\ddot ob}_n$ of infinitesimal isometries. It consists of vector fields of the form~\eqref{eqn:mob} with $b = 0$ and $A$ skew-symmetric. We have the following sequence of linear maps:
$$
0 \to \mathfrak{iso}_n \xrightarrow{i} \mathfrak{m\ddot ob}_n \xrightarrow{\mathrm{div}} \mathfrak{l}_n \to 0
$$
where $i \colon  \mathfrak{iso}_n \to \mathfrak{m\ddot ob}_n$ is the inclusion mapping, and $\mathfrak{l}_n$ is the $(n+1)$-dimensional vector space of (inhomogeneous) linear functions on the affine space $\R^n$. Since every divergence-free M\"{o}bius vector field is an isometry (preserving angles and volume implies preserving the metric), and every linear function can be obtained as the divergence of a M\"{o}bius vector field (which follows from transitivity of action of isometries on linear functions and can also be seen from explicit expression \eqref{eqn:mobdiv}), this sequence is exact. Therefore, any linear function on $\mathfrak{m\ddot ob}_n$ which vanishes on  $\mathfrak{iso}_n$ is of the form $f(\mathrm{div}\, \xi)$, where $f \in  \mathfrak{l}_n^*$ (here  $\mathfrak{l}_n^*$ is the dual space of $ \mathfrak{l}_n$). This in particular applies to the function
$
\xi \mapsto \nabla_\xi \log \mathrm{vol}(P) 
$
(which vanishes on isometries since isometries preserve the volume). So, there is $f \in  \mathfrak{l}_n^*$ such that
$$
\nabla_\xi \log \mathrm{vol}(P) = f(\mathrm{div}\,\xi).
$$
for all $\xi \in \mathfrak{m\ddot ob}_n$. Further observe that for $\xi$ of the form $\dot x = \lambda x$ (i.e. a homothety) one has
$
\nabla_\xi \log \mathrm{vol}(P) = n \lambda$ and $\mathrm{div}\,\xi = n \lambda.
$
Therefore, for any constant function $c \in  \mathfrak{l}_n$ one has $f(c) = c$. But any linear function $f \colon \mathfrak{l}_n \to \R$ which takes every constant to itself is of the form $f(l) = l(x)$ for some $x \in \R^n$. Furthermore, such $x$ is clearly unique, since evaluation at different points gives different functions on $\mathfrak{l}_n$  Denoting that $x$ by $\mathrm{m}(P)$, we get the result.
\\ \\
2. This follows from the invariance of all involved objects under similarities. 
\\ \\
3. This follows from additivity of the volume function and linearity of $\mathrm{div}\, \xi$ for a M\"{o}bius vector field $\xi \in  \mathfrak{m\ddot ob}_n$:
\begin{gather}
\nabla_\xi \log \mathrm{vol}(\Delta_i) = \mathrm{div}\, \xi (\mathrm{m}(\Delta_i)) \quad \Rightarrow \quad \nabla_\xi  \mathrm{vol}(\Delta_i) = \mathrm{vol}(\Delta_i) \,\mathrm{div}\, \xi (\mathrm{m}(\Delta_i))\\
 \Rightarrow \quad \nabla_\xi  \mathrm{vol}(P) = \sum    \nabla_\xi  \mathrm{vol}(\Delta_i)  = \sum  \mathrm{vol}(\Delta_i) \,\mathrm{div}\, \xi (\mathrm{m}(\Delta_i)) \\ \Rightarrow \quad    \nabla_\xi  \ln \mathrm{vol}(P)  = \sum  \frac{\mathrm{vol}(\Delta_i)}{\mathrm{vol}(P) } \,\mathrm{div}\, \xi (\mathrm{m}(\Delta_i)) = \mathrm{div}\, \xi \left(\sum  \frac{\mathrm{vol}(\Delta_i)}{\mathrm{vol}(P) } \, \mathrm{m}(\Delta_i)\right).
\end{gather}
\\ \\
4. Let $\Delta = (v_0, \dots, v_n) \in \mathcal P(\R^n)$ be a simplex such that $ \mathrm{vol}(\Delta) \neq 0$, and let \begin{equation}\label{pdef} p: =  -n\cdot \mathrm{m}(\Delta) +  \sum_{i=0}^n v_i.\end{equation} We will first show that $p$ is the circumcenter of $\Delta$. In view of Part 2 of the theorem, it suffices to consider the case $v_0 = 0$. Let  $\xi \in \mathfrak{m\ddot ob}_n$ be of the form
$$
\dot x =  |x|^2b - 2\langle b, x \rangle x,
$$
and let $\phi_t$ be a family of M\"{o}bius transformations integrating $\xi$. Note that since $\xi$ vanishes at the origin, we have $\phi_t(v_0) = \phi_t(0) = 0$. Therefore,
\begin{gather}
\nabla_\xi \log \mathrm{vol}(\Delta) = \left. \frac{d}{dt} \right\vert_{t=0} \!\!\!\!\!\!\!\log |\phi_t( v_1), \dots, \phi_t(v_n)|  = \frac{1}{D} \sum_{i=1}^n |v_1, \dots, |v_i|^2b - 2\langle b, v_i \rangle v_i, \dots, v_n| 
=  \ell(b) - 2  \langle b, \sum_{i=1}^n v_i \rangle,
\end{gather}
where
\begin{equation}\label{ldef}
D:= |v_1, \dots, v_n|, \quad \ell(b) := \frac{1}{D} \sum_{i=1}^n |v_1, \dots, b , \dots, v_n||v_i|^2.
\end{equation}
On the other hand, we have $\nabla_\xi \log \mathrm{vol}(\Delta) = \mathrm{div}\, \xi (\mathrm{m}(\Delta)) $ and  $\mathrm{div}\, \xi = - 2n \langle b, x \rangle$, so by \eqref{mainFormula2} we have
\begin{equation}\label{mcimpl}
 -2n\langle b, \mathrm{m}(\Delta) \rangle =  \ell(b) -2\langle b,  \sum_{i=1}^n v_i \rangle
 \end{equation}
for any $b \in \R^n$, which, by \eqref{pdef}, is equivalent to
$
 \langle b, p \rangle = \frac{1}{2} \ell(b).
 $
In particular, by definition \eqref{ldef} of the function $\ell$ we get
 $$
 \langle p, v_i \rangle =  \frac{1}{2} \ell(v_i) = \frac{1}{2}\langle v_i, v_i \rangle.
 $$
 Therefore,
\begin{equation}\label{weq}
 |p - v_i|^2 = \langle p, p \rangle - 2 \langle p, v_i \rangle +  \langle v_i, v_i \rangle =  \langle p, p \rangle = |p - v_0|^2
 \end{equation}
 for all $i$. So indeed $p$ is the circumcenter of $\Delta$, as claimed.
 
 Now, let us show that  $\mathrm{m}(\Delta)$ is the circumcenter of the medial simplex of $\Delta$. In view of \eqref{pdef} and an already established fact that  $p$ is the circumcenter of $\Delta$, this also proves~\eqref{mcccm}.  
%
Let
$$
  v'_i := \frac{1}{n} (-v_i +  \sum_{j=0}^n v_j )
$$
be centroids of faces of $\Delta$ (equivalently, vertices of the medial simplex $\Delta'$). Then 
$$
|v_i' - \mathrm{m}(\Delta)|^2 =  \frac{1}{n^2}|-v_i +  \sum_{j=0}^n v_j - n \cdot \mathrm{m}(\Delta)|^2 =  \frac{1}{n^2}|p- v_i|^2.
$$
By \eqref{weq}, the latter quantity is independent of $i$, so $\mathrm{m}(\Delta)$ is equidistant from the points $v_0', \dots, v_n'$, as desired. 
\\ \\
5. This directly follows from the two previous statements.
\end{proof}

\begin{remark}\label{rm:vfccm}
Along the same lines one shows that for quadratic vector fields $\xi$ of the form

\begin{equation}\label{eqn:mob2}
\dot x = Ax + |x|^2b + c
\end{equation}
one has
$
\nabla_\xi \log \mathrm{vol}(P) = \mathrm{div}\, \xi (\mathrm{ccm}(P)).
$
This gives a direct definition of the circumcenter of mass bypassing the notion of the M\"{o}bius center. We chose not to pursue this approach since geometric interpretation of vector fields~\eqref{eqn:mob2} is unclear. In a certain sense such vector fields interpolate between  M\"{o}bius fields \eqref{eqn:mob} and projective fields
$
\dot x = Ax + \langle b, x \rangle x + c,
$
just like the circumcenter of mass ``interpolates'' between the center of mass and the M\"{o}bius center.

\end{remark}


{
\bibliographystyle{plain}
\bibliography{ccm.bib}

\end{document}